\documentclass[12pt]{article}
\usepackage{graphicx}
\usepackage{subcaption}
\usepackage{caption}
\usepackage{hyperref}

\usepackage{amsmath,amsthm,amssymb,amsfonts,mathtools,eurosym}

\newtheorem{theorem}{Theorem}
\newtheorem{lemma}[theorem]{Lemma}
\usepackage{multirow}
\usepackage[ruled,linesnumbered]{algorithm2e}
\DeclareMathOperator*{\argmax}{arg\,max}
\DeclareMathOperator*{\argmin}{arg\,min}

\title{A Stochastic Conjugate Subgradient Algorithm for Two-stage Stochastic Programming}

\author{%
  Di Zhang \thanks{
  Department of Industrial and System Engineering,
  University of Southern California,
  Los Angeles, CA,
  \texttt{dzhang22@usc.edu}}
  \ and Suvrajeet Sen \thanks{
  Department of Industrial and System Engineering,
  University of Southern California,
  Los Angeles, CA, \texttt{suvrajes@usc.edu}
  }
  }

\begin{document}
\maketitle
\begin{abstract}%
Stochastic Optimization is a cornerstone of operations research, providing a framework to solve optimization problems under uncertainty. Despite the development of numerous algorithms to tackle these problems, several persistent challenges remain, including: (i) selecting an appropriate sample size, (ii) determining an effective search direction, and (iii) choosing a proper step size. This paper introduces a comprehensive framework, the Stochastic Conjugate Subgradient (SCS) framework, designed to systematically address these challenges. Specifically, The SCS framework offers structured approaches to determining the sample size, the search direction, and the step size. By integrating various stochastic algorithms within the SCS framework, we have developed a novel stochastic algorithm for two-stage stochastic programming. The convergence and convergence rates of the algorithm have been rigorously established, with preliminary computational results support the theoretical analysis. 
\end{abstract}

\section{Introduction} 
In this paper, we focus on the task of solving a two-stage Stochastic Programming (SP) with equality constraints
\begin{equation} \label{2SP}
\min_{x \in X \subseteq \mathbb{R}^{n_1}} f(x) = c(x)+E_{\omega}[h(x,\omega)],
\end{equation}

\noindent where $X=\{x: Ax = b, x \geq 0 \}$ is a polyhedral set, $c(\cdot)$ is differentiable convex deterministic function, $\omega \in \mathbb{R}^m$ is a continuous random vector and $h(\cdot,\omega)$ is non-smooth and convex random function. Here, $h(x,\omega)$ is called the recourse function and its value is usually obtained by solving another constrained optimization problem:
\begin{equation*}
    \begin{aligned}
        h(x,\omega) \triangleq & \ \min_{y} g(y)\\
        & \textit{s.t.} \quad  Dy=\xi(\omega)-C(\omega)x,\\
        & \quad \quad \ y \geq 0, y \in \mathbb{R}^{n_2}, 
    \end{aligned}
\end{equation*}

\noindent where $g(y)$ can be either a linear function $d^T y$ or a convex quadratic function $\frac{1}{2} y^T P y + d^T y$. For this reason, we will call $x$ and $y$ the first-stage and second-stage decision variable, respectively. For convenience, we will let $Y(x,\omega) \triangleq \{0 \leq y \in \mathbb{R}^{n_2} | Dy = \xi(\omega) - C(\omega)x \}$ through out the paper. 

It is well known that one can solve \eqref{2SP} by adopting first-order methods such as stochastic gradient descent (SGD)~\cite{RM1951,BB2011,SS2011} and stochastic decomposition~ \cite{HS1991,HS1994} where the subgradient calculations can be carried out very rapidly and in parallel (for a given $x$). However, first-order methods are provably inaccurate in many cases~\cite{L1984}. 

We propose a new algorithm that will accommodate both the curvature of functions, as well as the decomposition of subgradients by data points, as is common in stochastic programming. The algorithm includes the computational power of non-smooth conjugate subgradient method~\cite{W1975,zhang2024}, sequential sampling~\cite{Shapiro2003,peng2024asymmetric,liu2024mt2st}, decomposition~\cite{HS1991,HS1994} and active set method~\cite{N2006} to provide both computational reliability as well as a well-defined convergence rate. The convergence and convergence rate of the new method are established, with some experiments showing performance of the new algorithm. However, our convergence results are restricted to the linear equality constraints (i.e., $Ax =b$). A more general model with inequality (i.e., $x \geq 0$) will be included in a revised version. 

The structure of this paper is as follows. We begin by setting our work
within the context of the literature in §4.2. In §4.3, the discussion evolves
from mathematical concepts to computational algorithms. We will discuss the stochastic conjugate subgradient (SCS)
algorithm for solving large scale SP of the form \eqref{2SP} via a combination of conjugate gradient descent (CGD, \cite{W1975, powell1984nonconvex}) and sample average approximation (SAA), which
helps leverage a successive sampling methodology. This combination not only
introduces some elements of recursive computations but also bypasses the need for constrained SP algorithms or even the need for solving a
large linear system (of equations) rapidly. In §4.4, we will present the convergence analysis of SCS by showing that the algorithm guarantees the objective function values to form a sequence of
supermartingales and it converges in expectation. In §4.5, we will provide some technical details about its implementation and present computational results by comparing the SGD algorithm, the stochastic mirror descent (SMD)~\cite{nemirovski2009robust}, and the SCS algorithm. In the context of that comparison, we will
discuss some advantages and disadvantages of adopting the SCS algorithm in the context of specific SP.

\section{Connections with the literature} \label{LitReview}

First-order methods such as stochastic gradient descent (Robbins and Monro ~\cite{RM1951}; Bottou and Bousquet~\cite{BB2011}; Shalev-Shwartz~\cite{SS2011}) or dual coordinate ascent (Platt~\cite{P1998}; Joachims \cite{J1999}) are very popular for solving SP. Simple first-order gradient-based methods dominate the field for convincing reasons: low computational cost, simplicity of implementation, and strong empirical results.

For example, SGD finds an $\varepsilon$-suboptimal solution in time $O(1/\varepsilon)$. This runtime does not depend on the dimension of the decision variable and therefore is favorable for large-scale SP. However, the SGD approach has several limitations. (a) It lacks a clear stopping criterion; (b) It tends to be overly aggressive at the beginning of the optimization process; (c) While SGD reaches a moderate accuracy quite fast, its convergence becomes rather slow when we are interested in more accurate solutions~\cite{SZ2013}. 

For dual coordinate ascent, several authors (e.g., Mangasarian and Musicant~\cite{MM1999}; Hsieh et al.~\cite{H2008}) proved a linear convergence rate. 
However, there are two handicaps with dual coordinate ascent. First, the linear convergence parameter may be very close to zero and the initial unspecified number of iterations might be very large. Second, the analysis only deals with the sub-optimality of the dual objective, while our real goal is to bound the sub-optimality of the primal objective. Recently, Shalev-Shwartz and Zhang~\cite{SZ2013} proposed a stochastic version of dual coordinate ascent, in which at each round they choose one coordinate to optimize uniformly at random. However, such dual algorithms make it difficult to predict the quality of the primal solution even when the duality gap is considered to be small.  

Yet beyond first-order (BFO) methods are rarely used in SP, despite having several strengths: faster per-iteration convergence, frequent explicit regularization on step-size, and better parallelization than first-order methods.

In this paper, we present a decomposition-based conjugate subgradient algorithm that integrates data point decomposition, commonly used in stochastic programming, with Wolfe's conjugate subgradient method~\cite{W1975} for non-smooth convex SP. Our computational results demonstrate that the proposed method consistently yields solutions with lower objective function values from an optimization standpoint. Our approach combines non-smooth conjugate gradient descent (CGD)~\cite{W1975, N2006}, stochastic programming, and the active set method. To facilitate this integration, we modify Wolfe's algorithm from a direction-finding method to a stochastic conjugate subgradient (SCS) algorithm, where stochastic subgradients provide piece-wise linear approximations similar to stochastic decomposition (SD)~\cite{HS1994}. Direction-finding and line-search are then applied to these subgradients, projected onto an active set.

\section{Stochastic Conjugate Subgradient Algorithm}\label{SCSA}

In this section, decomposition refers to the separate treatment of the piecewise linear (max) function and the quadratic part in \eqref{2SP} to find the direction of change for the decision variables. We will first present the SCS algorithm (Algorithm \ref{SCS-2SP}). The main components of the algorithm include four key elements:
\begin{itemize}
\item Sequential function approximation. In many cases we have a fixed number of data points. However, our focus is on situations where the data can be queried from an oracle sequentially. As a result, we will not fix the value of $m$ in \eqref{2SP}. Instead, we use an adaptive sampling approach~\cite{M1999,Shapiro2003,lin2024economic,zhang2025alternating} to approximate function $f$ using $f_k$, where 
\begin{equation} \label{OSAA 2SP}
    f_k(x)=c(x) + \frac{1}{|S_k|} \sum_{i=1}^{|S_k|} h(x,\omega_i),
\end{equation}
Using Hoeffding's Inequality~\cite{H1994}, we can approximate (\ref{2SP}) by (\ref{OSAA 2SP}) to an arbitrarily high level of accuracy as the number of samples increases.

\item Direction Modification. Since there are linear constraints in \eqref{OSAA 2SP}, $g_k \in \partial f_k(x_k)$ (We will talk about how to obtain $g_k$ for specific problems in \$\ref{SQP}.) and the previous search direction $\tilde{d}_{k-1}$ may be infeasible. Thus, we will modify the subgradient $g_k$ and the search direction $\tilde{d}_{k-1}$ to ensure they are feasible search directions. Specifically, let the null-space $Z$ of $A$ be such that $\langle A, Z\rangle = 0$ and the modified directions are give by

\begin{equation} \label{d_k m}
    \tilde{g}_k = Z Z^T g_k \quad \textit{and} \quad \tilde{d}_{k-1}=Z Z^T \tilde{d}_{k-1}.
\end{equation}

\item Conjugate direction finding. The idea is inspired by Wolfe's non-smooth conjugate subgradient method, which uses the smallest norm of the convex combination of the previous modified search direction $\tilde{d}_{k-1}$ and the current modified subgradient $\tilde{g}_k$. More specifically, we first solve the following one-dimensional QP problem:

\begin{equation} \label{sp: p^k}
\lambda_k^* = \argmin_{\lambda \in [0,1]}\frac{1}{2}||\lambda \cdot (-\tilde{d}_{k-1}) + (1-\lambda) \cdot \tilde{g}_k||^2.
\end{equation}

\noindent Then we can set the new search direction as
\begin{equation} \label{d_k}
    \tilde{d}_k=-\big[\lambda_k^* \cdot (-\tilde{d}_{k-1})+(1-\lambda_k^*) \cdot \tilde{g}_k\big] := - Nr(\tilde{G}_k),
\end{equation}

\noindent where $\tilde{G}_k=\{-\tilde{d}_{k-1},\tilde{g}_k\}$, and $\lambda_k^*$ denotes the optimal weight for the convex combination.  Clearly if one fixes $\lambda = 0$, then the search direction reduces to that of the projected subgradient method.

\item Choice of step size. This is achieved by Wolfe's line search methods~\cite{W1975}. Let $g(t) \in \partial f_k(x_k + t \cdot \tilde{d}_k)$ and $\tilde{g}(t) = Z Z^T g(t)$. Define the intervals $L$ and $R$
\begin{equation} \label{sp: sdc}
    \begin{aligned}
        & L=\{t>0 \ | \ f_k(x_k + t \cdot \tilde{d}_k)-f_k(x_k) \leq -m_2 ||\tilde{d}_k||^2 t \},\\
        & R =\{t>0 \ | \ 0 > \langle \tilde{g}(t),\tilde{d}_k \rangle \geq -m_1 ||\tilde{d}_k||^2 \},
    \end{aligned}
\end{equation}
\noindent where $m_2<m_1<1$. The output of the step size will satisfy two metrics: (i) Identifying a set $L$ which includes points that sufficiently reduce the objective function value; (ii) Identifying a set $R$ for which the directional derivative estimate is improved. The algorithm seeks points which belong to $L \cap R$. Note that replacing $\tilde{d}_k$ with an arbitrary subgradient modification $\tilde{g}_k = Z Z^T g_k$ may not ensure sufficient descent unless the function $f_k$ is smooth.
\end{itemize}

\begin{algorithm}[h]
\small
\caption{Stochastic Conjugate Subgradient (SCS) Algorithm}
\label{SCS-2SP}

\SetAlgoLined
\nl $\varepsilon > 0$, $\tau>0$, $\delta_0$, $\eta_1 > 1$, $\eta_2 > 0$, $\gamma>1$ and $k \leftarrow 0$. \label{setup}

\nl Obtain the null-space $Z$ of $A$.

\nl Randomly generate $S_0$ samples from the data set.

\nl Set a feasible solution $\hat{x}_0 \in \mathbb{R}^{|S_0|}$ and an initial direction $d_0 \in \mathbb{R}^{|S_0|}$.

\nl $f_0(x)= c(x) + \frac{1}{|S_0|} \sum_{i=1}^{|S_0|} h(x,\omega_i).$ \label{SAA}

\While{$||d_k||>\varepsilon$}
{
    \nl $k=k+1$\;
   
    \nl Obtain $g_{k} \in \partial f_{k-1}(\hat{x}_k)$. \;

    \nl Modify the subgradient $g_k$ and the direction $\tilde{d}_{k-1}$ by \eqref{d_k m}.\;

    \nl Calculate the new current search direction $\tilde{d}_k$ by \eqref{sp: p^k} and \eqref{d_k}.\;
   
    \nl Apply line-search algorithm to find step size $t_k$ and set $x_{k+1}=x_{k}+t_k\tilde{d}_k$. \;
   
    \nl Randomly generate a set of new samples $s_k$ and let $S_k \triangleq S_{k-1} \cup s_k$.\;

    \nl Construct $f_k(x)=c(x) + \frac{1}{|S_k|} \sum_{i=1}^{|S_k|} h(x,\omega_i)$.\;
    
    \nl Randomly generate a set of new samples $T_k$ of cardinality $|S_k|$ independent of $S_{k}$.\;
    
    \nl Construct $\hat{f}_k(x)=c(x) + \frac{1}{|T_k|} \sum_{i=1}^{|T_k|} h(x,\omega_i)$.\;\label{SAA2}
    
    \eIf{$f_k(x_k)-f_k(\hat{x}_{k-1}) \leq \eta_1(\hat{f}_{k}(x_k)-\hat{f}_{k}(\hat{x}_{k-1}))$ \textit{and} $||d_k||>\eta_2\delta_k$}{\nl
        $\hat{x}_k \leftarrow x_k, \quad \delta_k \leftarrow \min \{\gamma\delta_{k-1},\frac{\varepsilon}{\xi}\}$\;
    }{\nl
            $\hat{x}_k \leftarrow \hat{x}_{k-1}, \quad \delta_k \leftarrow \frac{\delta_{k-1}}{\gamma}$\;
        } 
}
\end{algorithm}

\noindent Remark: (i) The initial choice of the maximal region $\delta_0$ may require some tuning depending on the dataset but it does not affect the convergence property of this algorithm. (ii) Based on Theorem \ref{sp: L^0} in §\ref{SQP}, $S_k$ should be chosen such that $|S_k| \geq -8\log(\varepsilon/2)\cdot \frac{(M-m)^2}{\kappa^2\delta_k^4}$.

\section{SCS for Two-stage Stochastic Quadratic Programming} \label{SQP}
In this section, we will use the SCS algorithm to solve \eqref{2SP} when the first-stage is a constrained quadratic program and the second-stage is a constrained Linear program (SQLP) and when the first-stage is a constrained quadratic program and the second-stage is a constrained quadratic program (SQQP). In the mathematical formulations of two-stage SQLP problems and two-stage SQQP problems, we use the subscript QL and QQ to identify SQLP and SQQP problems, respectively. Let $x$ and $y$ denote the first-stage and the second-stage decision variables, with x belonging to the set $X \subseteq \mathbb{R}^{n_1}$ and $y$ belonging to a polyhedron in $\mathbb{R}^{n_2}$. The mathematical formulation of a two-stage SQLP is given below.

\begin{equation*}
    \begin{aligned}
      & \min \ f(x) \triangleq \frac{1}{2} x^T Q x + c^T x + E[h(x,\omega)] \\
      & \textit{s.t.} \quad  x \in X=\{x: Ax = b\} \in \mathbb{R}^{n_1}, 
    \end{aligned}
\end{equation*}
\noindent where $h(x,\omega)$ is defined as 

\begin{equation} \label{h_L}
    \begin{aligned}
        h(x,\omega)=h_{QL}(x,\omega) \triangleq & \min \ d^Ty\\
        & \textit{s.t.} \quad  Dy=\xi(\omega)-C(\omega)x,\\
        & \quad \quad \ y \geq 0, y \in \mathbb{R}^{n_2}. 
    \end{aligned}
\end{equation}

\noindent Here $A \in  \mathbb{R} ^{m_1 \times n_1}$ is a deterministic matrix,
and $D \in  \mathbb{R}^{m_2 \times n_2}$ is a deterministic matrix. In addition, $\xi (\omega)$ denotes a random vector, $C(\omega)$ denotes a random matrix, and $E [\cdot ]$ denotes the expectation with respect to the probability measure of $\omega$. Finally, we assume that the second-stage cost vector $d$ is fixed and $Q$ to be a positive definite matrix in SQLP/SQQP problems. A specific example of SQLP problem can be the kernel SVM problem where $h_{QL}(x,\omega) \triangleq \max\{0,1-\omega\langle x, Q_i \rangle \}$ is the hinge-loss function.

\noindent To get a subgradient of \eqref{h_L}, we derive its dual for a specific $\bar{x}$~\cite{liu2020asymptotic} as follows:
\begin{equation*}
    \begin{aligned}
        h_{QL}(\bar{x},\omega) = & \max \ \pi^T(\xi(\omega)-C(\omega)\bar{x})\\
        & \textit{s.t.} \ \pi \in \Pi = \{\pi: D^T\pi \leq d \}.
    \end{aligned}
\end{equation*}
\noindent Let 

\begin{equation*}
    \pi^*=\argmax_{\pi \in \Pi} \pi^T(\xi(\omega)-C(\omega)\bar{x})
\end{equation*}

\noindent Then a subgradient of $\upsilon_h \in \partial_{x} h(x,\omega)$ is

\begin{equation*}
    \upsilon_h=-(\pi^*)^TC(\omega).
\end{equation*}

As for a two-stage SQQP, the mathematical formulation is:

\begin{equation*}
    \begin{aligned}
      & \min \ f(x) \triangleq \frac{1}{2} x^T Q x + c^T x + E[h(x,\omega)] \\
      & \textit{s.t.} \quad  x \in X=\{x: Ax = b\} \in R^{n_1}, 
    \end{aligned}
\end{equation*}

\noindent where $h(x,\omega)$ is defined as
\begin{equation} \label{QQ}
    \begin{aligned}
        h(x,\omega)=h_{QQ}(x,\omega) \triangleq & \min_{0 \leq y \in \mathbb{R}^{n_2}} \frac{1}{2} y^T P y+d^Ty\\
        & s.t. \ Dy=\xi(\omega)-C(\omega)x
    \end{aligned}
\end{equation}

\noindent To get a subgradient of \eqref{QQ}, we derive its dual for a specific $\bar{x}$~\cite{liu2020asymptotic} as follows:

\begin{equation}
    h_{QQ}(x,\omega) = \max_{0 \leq s \in R^{n_2}} -\frac{1}{2}s^T H s + e(x,\omega)^T s, 
\end{equation}

\noindent where 

\begin{equation*}
    \begin{aligned}
        & M=DP^{-1/2}, \ H=P^{-1/2}(I-M^T(MM^T)^{-1}M)P^{-1/2},\\
        & e(x,\omega)=Hd-P^{-1/2}(MM^T)^{-1}(\xi(\omega)-c(\omega)\bar{x}).
    \end{aligned}
\end{equation*}

\noindent Let 

\begin{equation*}
s^*=\argmax_{0 \leq s \in R^{n_2}}-\frac{1}{2}s^T H s + e(\bar{x},\omega)^T s=Proj_{s \geq 0}\{s=H^{-1}e(\bar{x},\omega)\}
\end{equation*}

\noindent Then a subgradient of $g_h \in \partial_{x} h(x,\omega)$ is

\begin{equation*}
    g_h=P^{-1/2}(MM^T)^{-1}c(\omega)s^*.
\end{equation*}

\noindent Given this subgradient of $h(x,\omega)$, it follows that a subgradient of $f(x)$ is

\begin{equation*}
    g=Qx+c+g_h.
\end{equation*}

\section{Convergence Properties}

The following assumptions, which are common for two-stage SLP models, will be imposed throughout the convergence rate analysis.

\begin{itemize} \label{conditions}
    \item A1: The optimization problem \eqref{2SP} is convex and finite-valued. Thus, extended real-valued functions are not permissible in our setting. As a consequence,  the function $f$ has a Lipschitz constant $0 < L < \infty$ and the subdifferential is non-empty at every $x$.

    \item A2: The relatively complete recourse assumption holds, i.e., for every $\omega_i$ and $x \in \mathcal{X}$, we have the second stage problem is bounded,
    \begin{equation*}
        m \leq h(x,\omega_i) \leq M.
    \end{equation*}
\end{itemize}

\subsection{Sample Complexity Analysis}

\begin{theorem} \label{sp: L^0} Given any $0 < \varepsilon <1$, there exists a constant $\kappa > 0$ such that when 

\begin{equation} \label{sp: sample size}
    |S_k| \geq -8\log(\varepsilon/2)\cdot \frac{(M-m)^2}{\kappa^2\delta_k^4},
\end{equation}
\noindent we have

\begin{equation*} \label{fl1}
    Pr\big (||f_k(x)-f(x)|| \leq \kappa\delta_k^2, \ \forall x \in \mathcal{B}(\hat{x}_k, \delta_k) \big ) \geq 1-\varepsilon.
\end{equation*}

\end{theorem}

\begin{proof}
First, for the fixed point $\hat{x}_k$, consider the following identities 
\begin{equation*}
    \begin{aligned}
        & f(\hat{x}_k)=c(\hat{x}_k)+E_{\omega}[h(\hat{x}_k,\omega)].\\
        & f_k(\hat{x}_k)=c(\hat{x}_k) + \frac{1}{|S_k|} \sum_{i=1}^{|S_k|} h(\hat{x}_k,\omega_i).
    \end{aligned}
\end{equation*}

\noindent From A2 we have $m \leq h(x,\omega_i) \leq M$. Applying Hoeffding's inequality, 

\begin{equation*}
    \mathbb{P}(|f_k(\hat{x}_k)-f(\hat{x}_k)| \leq \frac{1}{2}\kappa \delta_k^2) \geq 1-2\exp \left (-\frac{\kappa^2\delta_k^4 |S_k|^2}{8 \cdot |S_k| \cdot (M-m)^2} \right ) \geq 1-\varepsilon,
\end{equation*}
\noindent which indicates that when 
\begin{displaymath}
   |S_k| \geq -8\log(\varepsilon/2)\cdot \frac{(M-m)^2}{\kappa^2\delta_k^4}, 
\end{displaymath} 
we have 
\begin{equation*}
    \mathbb{P}(|f_k(\hat{x}_k)-f(\hat{x}_k)| \leq \frac{1}{2}\kappa\delta_k^2 ) \geq 1-\varepsilon.
\end{equation*}

\noindent For any other $x \in \mathcal{B}(\hat{x}_k,\delta_k)$, if $|f_k(\hat{x}_k)-f(\hat{x}_k)| \leq \frac{1}{2}\kappa\delta_k^2$, then

\begin{equation*}
    \begin{aligned}
        |f_k(x)-f(x)| & \leq |f_k(x)-f_k(x_k)|+|f_k(x_k)-f(x_k)|+ |f(x_k)-f(x)| \\
        & \leq 2L \cdot \delta_k + \frac{1}{2} \kappa \delta_k^2\\
        & \leq \kappa \delta_k^2,
    \end{aligned}
\end{equation*}

\noindent where the second last inequality is due to the assumption that $f_k$ and $f$ are Lipschitz continuous and the last inequality is because $\kappa>\frac{4L}{\delta_{min}}>\frac{4L}{\delta_k}$. Thus, we conclude that if $|S_k|$ satisfies Equation \eqref{sp: sample size}, then 

\begin{equation*} 
    \mathbb{P}(|f_k(x)-f(x)| \leq \kappa \delta_k^2, \ \forall x \in \mathcal{B}(\hat{x}_k,\delta_k) ) \geq 1-\varepsilon.
\end{equation*}
\end{proof}

\subsection{Feasible direction and sufficient decrease}
\begin{theorem} \label{d_k m is good}
In each step of algorithm \ref{SCS-2SP}, $\tilde{d}_k$ is a feasible direction.
\end{theorem}

\begin{proof}
The proof is well-documented in the literature on linearly constrained optimization (e.g., Nocedal and Wright~\cite{N2006}).
\end{proof}









\begin{theorem} \label{d_k m descent}
In each step of algorithm \ref{OSAA 2SP}, with assumption A2, there exists $t_k > 0$ such that
\begin{equation} \label{sufficient decrease}
    f_k(x_k+t_k \cdot \tilde{d}_k)-f_k(x_k) \leq -m_2 t_k ||\tilde{d}_k||^2.
\end{equation}
\end{theorem}

\begin{proof}
First, since $-\tilde{d}_k$ is the smallest norm in the convex hull of $\{\tilde{g}_i\}_{i=1}^{k}$, we have $\langle \tilde{d}_k, \tilde{g}_i \rangle \leq - ||\tilde{d}_k||^2$, we have 
\begin{equation} \label{d_k g_k}
   ||\tilde{d}_k||^2 \leq -\langle \tilde{d}_k, \tilde{g}_k \rangle 
\end{equation}
We will then divide the proof into two cases.

\noindent Case 1: if there exists $t_k$ such that $R$ conditions is satisfied, i.e., 

\begin{equation*} 
 \langle \tilde{g}(t_k),\tilde{d}_k \rangle \geq -m_1 ||\tilde{d}_k||^2.
\end{equation*}

\noindent According to \eqref{d_k g_k}, we can equivalently write 

\begin{equation*}
     \langle \tilde{g}(t_k),\tilde{d}_k \rangle - \langle \tilde{d}_k, \tilde{g}_k \rangle \geq \langle \tilde{g}(t_k),\tilde{d}_k \rangle + ||\tilde{d}_k||^2 \geq (1-m_1) ||\tilde{d}_k||^2.
\end{equation*}

\noindent Then, based on assumption A2, we have 

\begin{equation*}
    (1-m_1) ||\tilde{d}_k||^2  \leq \langle \tilde{g}(t_k),\tilde{d}_k \rangle - \langle \tilde{d}_k, \tilde{g}_k \rangle \leq t_k L ||\tilde{d}_k||^2.
\end{equation*}

\noindent Thus, we have $t_k \geq (1-m_1)/L$.

\noindent Case 2: if $t_k$ in case 1 makes some of the constraints in $l \notin J_k$ violated. Then, as $t_k$ gets larger from $0$, we can find $l_k$ such that 
\begin{equation*}
   a_{l_k}^T(x_k + t_k \tilde{d}_k)  = b_{l_k}. 
\end{equation*}

\noindent Then, 

\begin{equation*}
    t_k ||a_{l_k}|| \cdot ||\tilde{d}_k|| \geq t_k |a_{l_k}^T d_k| = |a_{l_k}^Tx_k-b_{l_k}| \geq \tau.
\end{equation*}

\noindent Thus, 
\begin{equation*}
    t_k \geq \frac{\tau}{||a_{l_k}|| \cdot ||\tilde{d}_k||}.
\end{equation*}
\end{proof}

\subsection{Convergence Rate and Optimality Condition}

The convergence rate analysis is similar to~\cite{zhang2024sampling,zhang2024stochastic}.  Interested readers can refer to the articles for the detailed proofs. For the sake of completeness, we will list the key theorems here.

\begin{theorem}
    Let $ T_{\varepsilon} = \inf \{ k \geq 0: ||\tilde{d}_k|| < \varepsilon\}$. Then $T_{\varepsilon}$ is a stopping time for the stochastic process $\hat{X}^k$ and 
    \begin{equation} \label{sp: convergence equation}
    \mathbb{E}[T_{\varepsilon}] \leq \frac{p}{2p-1}\Big(\frac{2 \zeta F_{max} }{\Theta \varepsilon^2}+2\Big).
\end{equation}
\end{theorem}

\noindent Theorem \ref{sp: optimality} demonstrates that if $||\tilde{d}_k||<\varepsilon$, the $\varepsilon$-optimality condition for $f(x)$ will be satisfied. The error bound on this optimality is linked to both the sample size and the subgradient of $f_k(x)$. 

\begin{theorem} \label{sp: optimality}
    Let $\tilde{d}^* = - \argmin_{ \tilde{g} = Z Z^T g, g \in \partial f(\hat{x}_k)} ||\tilde{g}|| $. If $k$ is the smallest index for which $||\tilde{d}_k|| < \varepsilon$ and $|S_k| \geq \frac{-2(\varepsilon')^2}{L^2\log \delta}$, then 
    \begin{equation*}
        P(||\tilde{d}^*|| < 4 \varepsilon + \varepsilon') \geq 1-\delta.
    \end{equation*}
\end{theorem}

\noindent To prove the theorem, we first define a few notations.

\begin{itemize}
    \item $\tilde{d}^* = - \argmin_{ \tilde{g} = Z Z^T g, g \in \partial f(\hat{x}_k)} ||\tilde{g}||$: negative of the smallest norm subgradient of $f(\hat{x}_k)$ projected onto the null-space $Z$.
    \item $\tilde{d}_k^* = - \argmin_{\tilde{g}_k = Z Z^T g, g \in \partial f_k(\hat{x}_k)} ||\tilde{g}_k|| $: negative of the smallest norm subgradient of $f_k(\hat{x}_k)$ projected onto the null-space $Z$.
    \item $d_k = -Nr(\{g_k,-d_{k-1} \})$, where $g_k \in \partial f_k(\hat{x}_k)$.
    \item Subdifferential of $f$ at $z'$: $\partial f(z') = \{q: f(z) \geq f(z') + q^T (z-z')\}$.
    \item $\varepsilon$-subdifferential of $f$ at $z'$: $\partial f_{\varepsilon}(z') = \{q: f(z) \geq f(z') + q^T (z-z')-\varepsilon\}$.
    \item Directional derivative of $f$ at $z$ in direction $d$: $f'(z;d)=\max\{d^Tq: q \in \partial f(z)\}$.
    \item $\varepsilon$-directional derivative of $f$ at $z$ in direction $d$: $f_{\varepsilon}'(z;d)=\max\{d^Tq: q \in \partial_{\varepsilon} f(z)\}$.
    \item $\tilde{\mathcal{X}}(z) := -\min_{\tilde{d} = Z Z^T d, ||\tilde{d}|| = 1} f'(z;\tilde{d})$ and $\tilde{\mathcal{X}}_{k}(z) := -\min_{\tilde{d} = Z Z^T d, ||\tilde{d}|| = 1} f_k'(z;\tilde{d})$.
    \item $\tilde{\mathcal{X}}_{\varepsilon}(z) := -\min_{\tilde{d} = Z Z^T d, ||\tilde{d}|| = 1} f_{\varepsilon}'(z;\tilde{d})$ and $\bar{d}_k(z) = - \argmin_{\tilde{d} = Z Z^T d, ||\tilde{d}|| = 1} f_{\varepsilon}'(z;\tilde{d}) $.
\end{itemize}

\begin{lemma} \label{sp: d_k^*}
    Let the tuple $(m_1, m_2)$ satisfy the requirement in \eqref{sp: sdc} ($\frac{1}{4} \leq m_1 < \frac{1}{2}$ and $\frac{1}{4} \leq m_2 < m_1$). If $k$ is the smallest index for which $||\tilde{d}_k|| < \varepsilon$, then we have $||\tilde{d}_k^*|| < 4 \varepsilon$.
\end{lemma}

\begin{proof}
    Suppose the claim is false, then by the definition of $\tilde{d}_k^*$, we have $||\tilde{g}_k|| \geq ||\tilde{d}_k^*|| \geq 4\varepsilon$. Thus,
    \begin{equation} \label{sp: norm of d_k}
        \begin{aligned}
            ||\tilde{d}_k||^2 & = ||\lambda_k^* Z Z^T g_k-(1-\lambda_k^*)Z Z^T d_{k-1}||^2\\
            & = (\lambda_k^*)^2||\tilde{g}_k||^2+(1-\lambda_k^*)^2||\tilde{d}_{k-1}||^2-2\lambda_k^*(1-\lambda_k^*)\langle \tilde{g}_k,\tilde{d}_{k-1} \rangle \\
            & \geq (\lambda_k^*)^2||\tilde{g}_k||^2+(1-\lambda_k^*)^2||\tilde{d}_{k-1}||^2 \\
        \end{aligned}
    \end{equation}

\noindent where the inequality holds because line-search algorithm ensures the condition R in \eqref{sp: sdc}. 

\noindent We will now divide the analysis into 2 cases which are examined below: a) $||\tilde{g}_k|| > ||\tilde{d}_{k-1}|| \geq \varepsilon$ and $||\tilde{g}_k|| \geq 4 \varepsilon$  and b) $||\tilde{d}_{k-1}|| \geq ||\tilde{g}_k|| \geq 4 \varepsilon$. 

(a) In this case from \eqref{sp: norm of d_k} we have 
\begin{equation} \label{sp: norm of d_k case 1}
    ||\tilde{d}_k||^2 \geq \big(17 (\lambda_k^*)^2 -2 \lambda_k^* + 1 \big) \varepsilon^2
\end{equation}

\noindent Note that
\begin{equation*}
    \lambda_k^*=\frac{\langle \tilde{g}_k,\tilde{d}_{k-1} \rangle + ||\tilde{g}_k||^2}{||\tilde{d}_{k-1}||^2+||\tilde{g}_k||^2+2\langle \tilde{g}_k,\tilde{d}_{k-1} \rangle} \geq \frac{-m_2||\tilde{d}_{k-1}||^2+||\tilde{g}_k||^2}{||\tilde{d}_{k-1}||^2+||\tilde{g}_k||^2},
\end{equation*}

\noindent where the inequality holds because the R condition in \eqref{sp: sdc} ensures $0 > \langle \tilde{g}_k,\tilde{d}_{k-1} \rangle \geq - m_2 ||\tilde{d}_{k-1}||^2$. Thus, we claim that $\lambda_k^* \geq \frac{2}{17}$ because it is suffice to show

\begin{equation*}
    - 17 \cdot m_2||\tilde{d}_{k-1}||^2+17 \cdot ||\tilde{g}_k||^2 \geq 2 ||\tilde{d}_{k-1}||^2 + 2 ||\tilde{g}_k||^2.
\end{equation*}

\noindent This can be verified by observing $m_2 < m_1 < \frac{1}{2}$ and $||\tilde{g}_k|| > ||\tilde{d}_{k-1}||$. Thus, based on \eqref{sp: norm of d_k case 1} and $\lambda_k^* \geq \frac{2}{17}$, we have
$||\tilde{d}_k|| \geq \varepsilon $. This contradicts the assumptions of the lemma. 

\noindent (b) If $||\tilde{d}_{k-1}|| \geq ||\tilde{g}_k|| \geq 4 \varepsilon$, then we have 
\begin{equation*}
    ||\tilde{d}_k||^2 \geq \big(32 (\lambda_k^*)^2 -32 \lambda_k^* + 16 \big) \varepsilon^2 \geq 8 \varepsilon^2. 
\end{equation*}

\noindent Thus, $||\tilde{d}_k|| \geq 2 \sqrt{2} \varepsilon $, which also contradicts the assumptions of the lemma.
\end{proof}

\begin{lemma} \label{sp: d and X}
    $||\tilde{d}^*|| = \tilde{\mathcal{X}}(\hat{x}_k)$ and $||\tilde{d}_k^*|| = \tilde{\mathcal{X}}_k(\hat{x}_k)$.
\end{lemma}

\begin{proof}
    According to the definition of $\tilde{\mathcal{X}}(\hat{x}_k)$,
    \begin{equation} \label{sp: min-max}
        \mathcal{X}(\hat{x}_k)=-\min_{\tilde{d} = Z Z^T d, ||\tilde{d}|| = 1} f'(\hat{x}_k;\tilde{d}) = -\min_{\tilde{d} = Z Z^T d, ||\tilde{d}|| = 1} \max_{q \in \partial f(\hat{x}_k)}d^Tq
    \end{equation}

\noindent Note that the min-max  problem \eqref{sp: min-max} has a stationary point

\begin{equation*}
    \bar{q}=\argmin_{\tilde{q} = Z Z^Tq, q \in \partial f(\hat{x}_k)} ||\tilde{q}||, \quad  \bar{d} = - \frac{\bar{q}}{||\bar{q}||}.
\end{equation*}

\noindent Thus, $\tilde{\mathcal{X}}(\hat{x}_k) = ||\bar{q}|| = ||\tilde{d}^*||$. Similarly, we can also prove that $ \tilde{\mathcal{X}}_k(\hat{x}_k) = ||\tilde{d}_k^*||$.
\end{proof}

\begin{theorem} \label{sp: X and X_k}
    If $|S_k| \geq \frac{-2(\varepsilon')^2}{L^2\log \delta}$, then $P(\tilde{\mathcal{X}}_{\varepsilon}(\hat{x}_k)-\tilde{\mathcal{X}}_k(\hat{x}_k) \leq \varepsilon') \geq 1-\delta$.
\end{theorem}

\begin{proof}
By definition, $\bar{d}_k(\hat{x}_k) = - \argmin_{||\tilde{d} = Z Z^T d, \tilde{d}|| = 1} f_{\varepsilon}'(\hat{x}_k;\tilde{d}) $. Hence,
    \begin{equation*}
        \begin{aligned}
        \tilde{\mathcal{X}}_{\varepsilon}(\hat{x}_k)-\tilde{\mathcal{X}}_{k}(\hat{x}_k)
        & \leq \frac{1}{|S_k|}\sum_{i=1}^{|S_k|} h'(\hat{x}_k;\bar{d}_k(\hat{x}_k),\omega_i)-E[h_{\varepsilon}'(\hat{x}_k;\bar{d}_k(\hat{x}_k),\omega)]\\
        & = \frac{1}{|S_k|}\sum_{i=1}^{|S_k|} [h'(\hat{x}_k;\bar{d}_k(\hat{x}_k),\omega_i)-h_{\varepsilon}'(\hat{x}_k;\bar{d}_k(\hat{x}_k),\omega_i)]\\
        & + \frac{1}{|S_k|}\sum_{i=1}^{|S_k|} h_{\varepsilon}'(\hat{x}_k;\bar{d}_k(\hat{x}_k),\omega_i) - E[h_{\varepsilon}'(\hat{x}_k;\bar{d}_k(\hat{x}_k),\omega)]
        \end{aligned}
    \end{equation*}
\noindent Note that for each $\omega_i$, we have $h'(\hat{x}_k;\bar{d}_k(\hat{x}_k),\omega_i) \leq h_{\varepsilon}'(\hat{x}_k;\bar{d}_k(\hat{x}_k),\omega_i)$, since $\partial h(\hat{x}_k,\omega_i) \subseteq \partial h_\varepsilon(\hat{x}_k,\omega_i) $. Also, by Hoeffding's inequality, if $|S_k| \geq \frac{-2(\varepsilon')^2}{L^2\log \delta}$, then 

\begin{equation*}
    P\big(\frac{1}{|S_k|}\sum_{i=1}^{|S_k|} h_{\varepsilon}'(\hat{x}_k;\bar{d}_k(\hat{x}_k),\omega_i) - E[h_{\varepsilon}'(\hat{x}_k;\bar{d}_k(\hat{x}_k),\omega)] \leq \varepsilon'\big) \geq 1-\delta,
\end{equation*}

\noindent which concludes the proof.
\end{proof}

\begin{proof} (Theorem \ref{sp: optimality})
    First, by Theorem \ref{sp: X and X_k} and Lemma \ref{sp: d and X}, if $|S_k| \geq \frac{-2(\varepsilon')^2}{L^2\log \delta}$, then with probability at least $1-\delta$,
    \begin{equation*}
        ||\tilde{d}^*||-||\tilde{d}_k^*||=\lim_{\varepsilon \to 0} \tilde{\mathcal{X}}_{\varepsilon} (\hat{x}_k) -\tilde{\mathcal{X}}_k(\hat{x}_k) \leq \varepsilon'.
    \end{equation*}
    \noindent Combining with Lemma \ref{sp: d_k^*}, we have
    \begin{equation*}
        P \big(||\tilde{d}^*|| < 4 \varepsilon + \varepsilon' \big) \geq 1 - \delta.
    \end{equation*}
\end{proof}

\section{Preliminary Computational Results} \label{Computational result: scs for sp}

Our computational experiments are based on data sets available at the USC 3D Lab~\footnote{https://github.com/USC3DLAB/SD}. Specifically, these problems are two-stage stochastic linear programming problems, and we have tested the LandS3, lgsc, pgp2, and ssn datasets. The study considers three different methods: the SGD~\cite{RM1951} algorithm with projection, stochastic mirror descent~\cite{nemirovski2009robust}, and the SCS algorithm. We track the decrease in objective values with respect to the number of iterations for all algorithms. The figures also include the 95\% upper and lower bounds of the objective values obtained from SD~\cite{HS1991,HS1994} as a benchmark. The algorithms were implemented on a MacBook Pro 2023 with a 2.6GHz 12-core Apple M3 Pro processor and 32GB of 2400MHz DDR4 onboard memory. The code used in this research is available at the following GitHub repositories: \href{https://github.com/yhz0/twosd-cpp}{SCS for SP} (accessed on August 24, 2024), and the computation results are shown in Figure \ref{obj scs for sp first 50} and Figure \ref{obj scs for sp last 50}.
.

\begin{figure}[ht]
\caption{$\{f(\hat{x}_k)\}_{k=1}^{k=50}$ for different combinations (data,algorithm).}
\centering
\begin{minipage}[t]{0.45\textwidth}
\centering
\includegraphics[width=7cm]{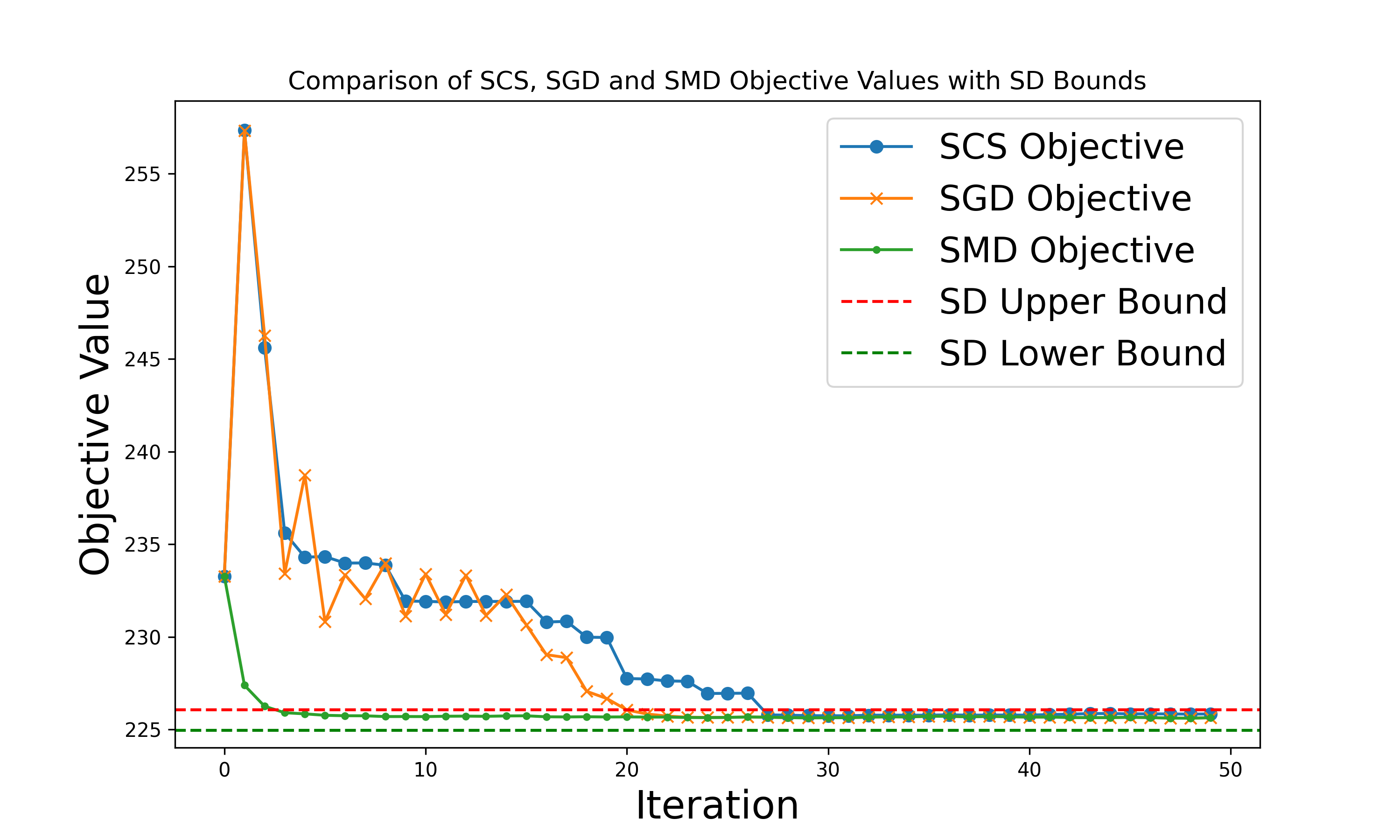}
\caption*{LandS3}
\end{minipage}
\begin{minipage}[t]{0.45\textwidth}
\centering
\includegraphics[width=7cm]{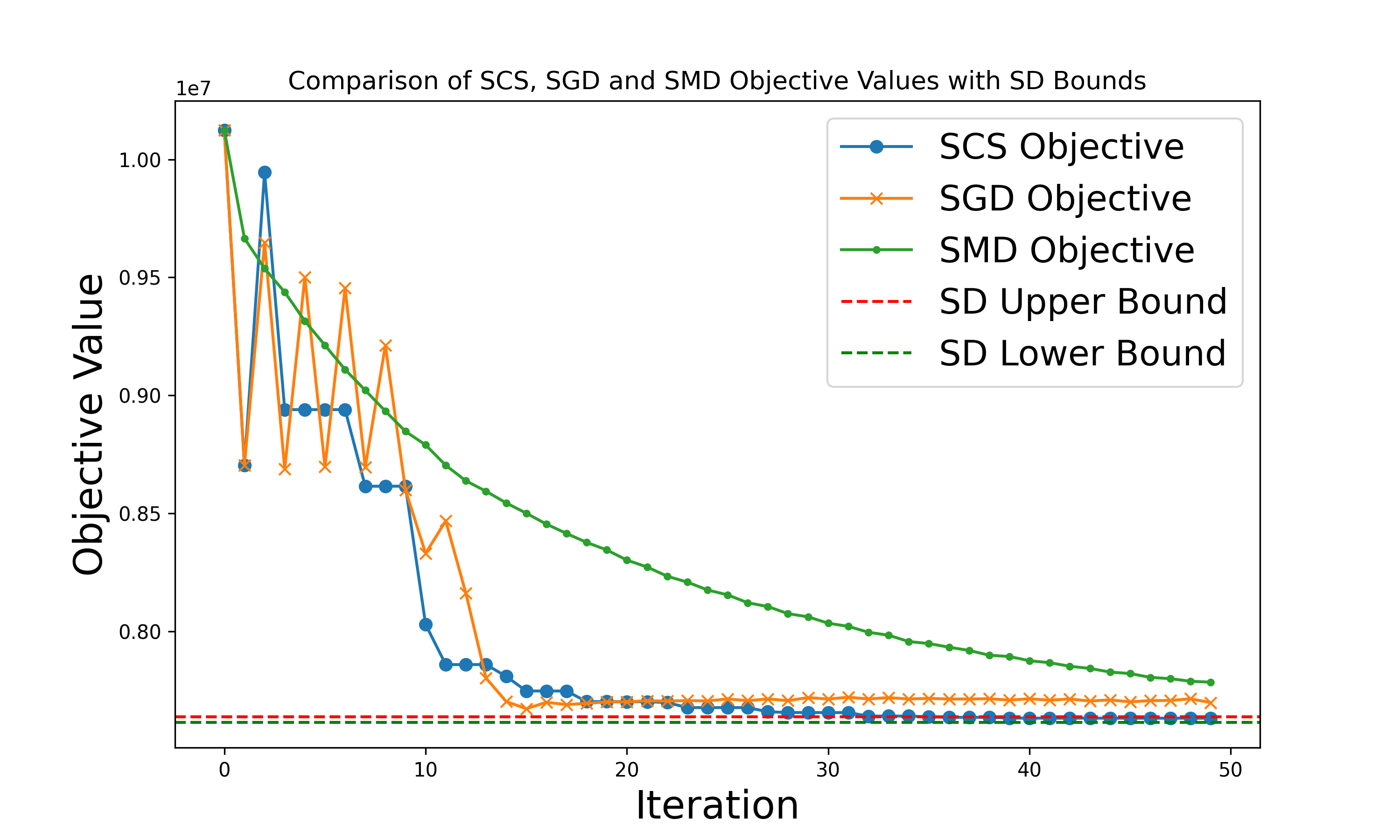}
\caption*{lgsc}
\end{minipage} \\

\begin{minipage}[t]{0.45\textwidth}
\centering
\includegraphics[width=7cm]{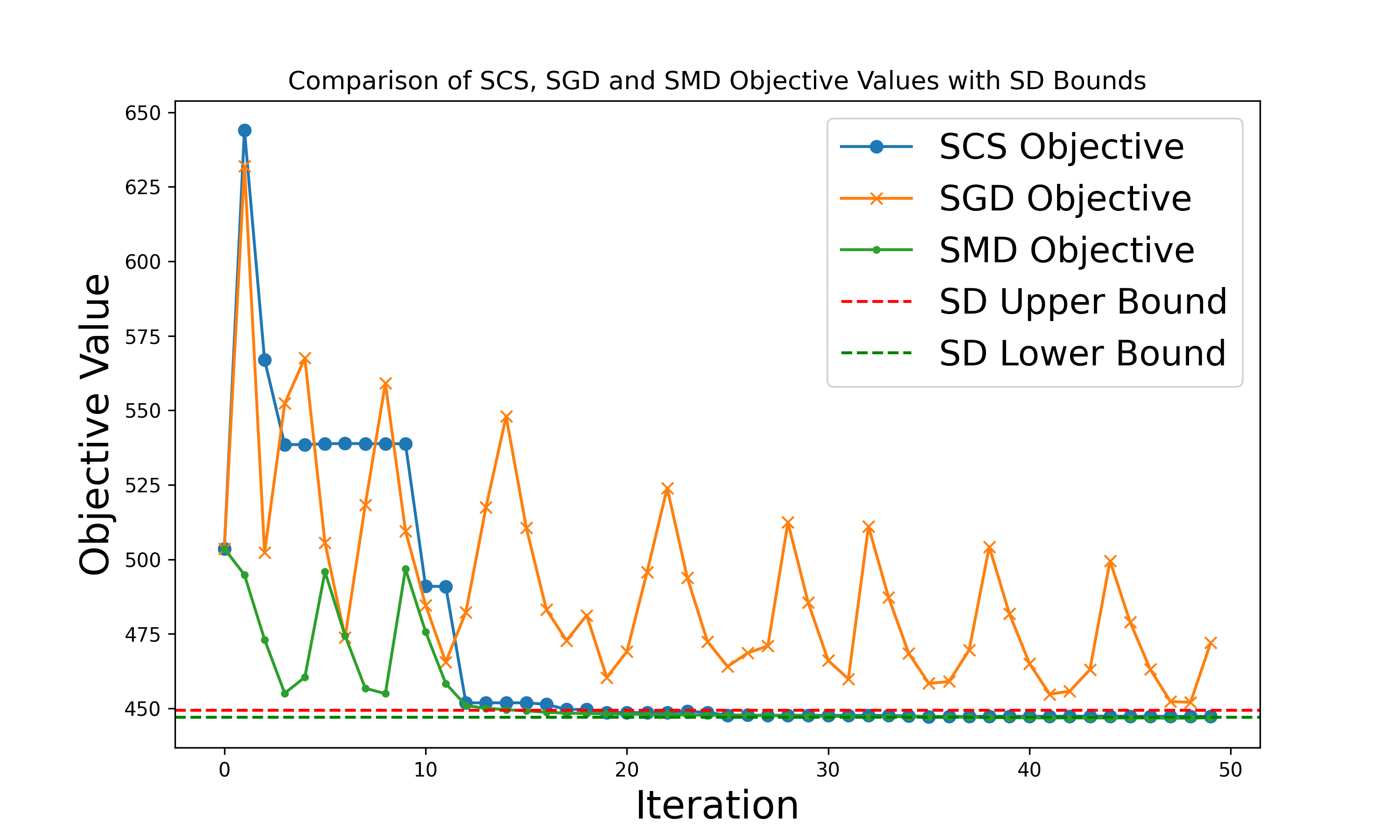}
\caption*{pgp2}
\end{minipage}
\begin{minipage}[t]{0.45\textwidth}
\centering
\includegraphics[width=7cm]{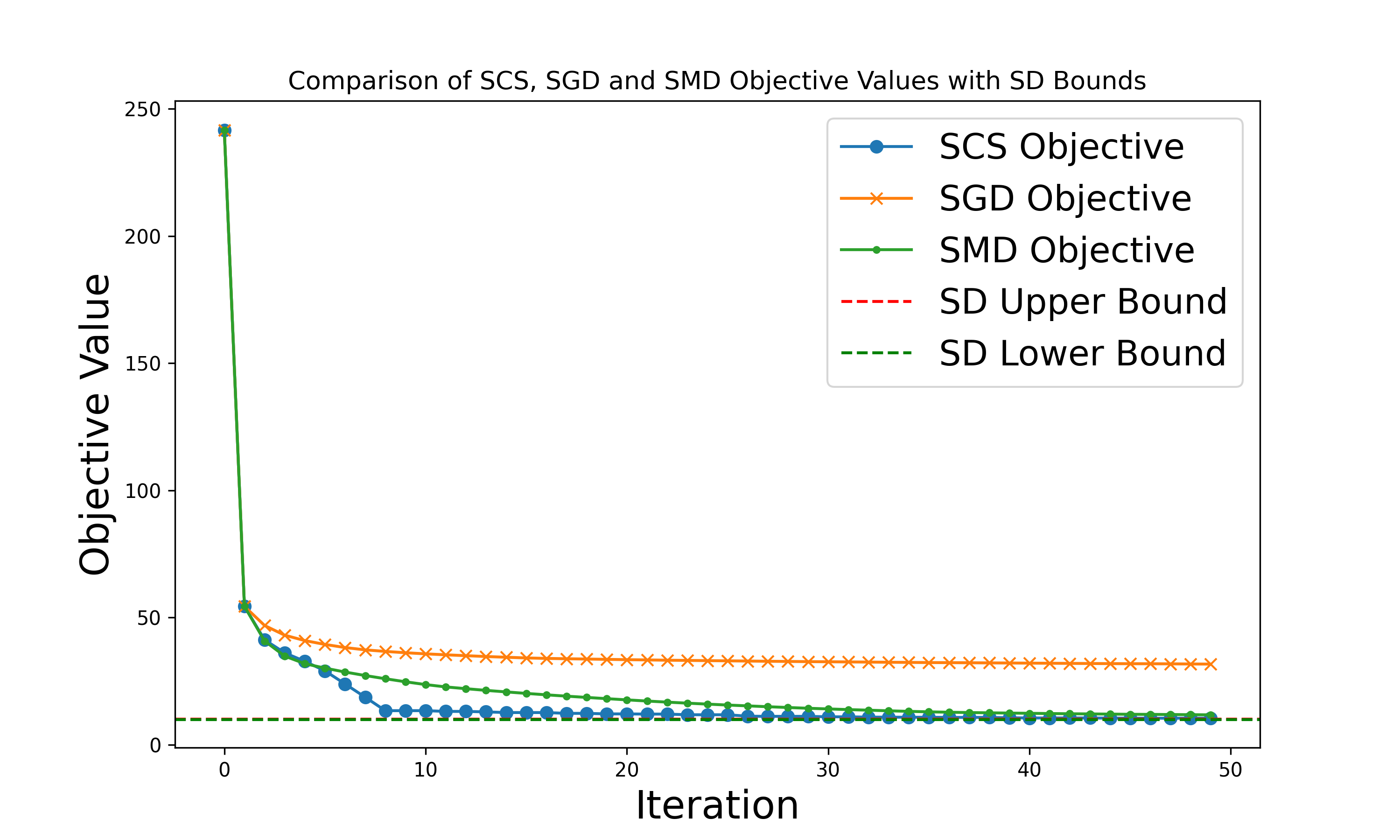}
\caption*{ssn}
\end{minipage}
\label{obj scs for sp first 50}
\end{figure}

\begin{figure}[ht]
\caption{$\{f(\hat{x}_k)\}_{k=-50}^{k=-1}$ for different combinations (data,algorithm).}
\centering
\begin{minipage}[t]{0.45\textwidth}
\centering
\includegraphics[width=7cm]{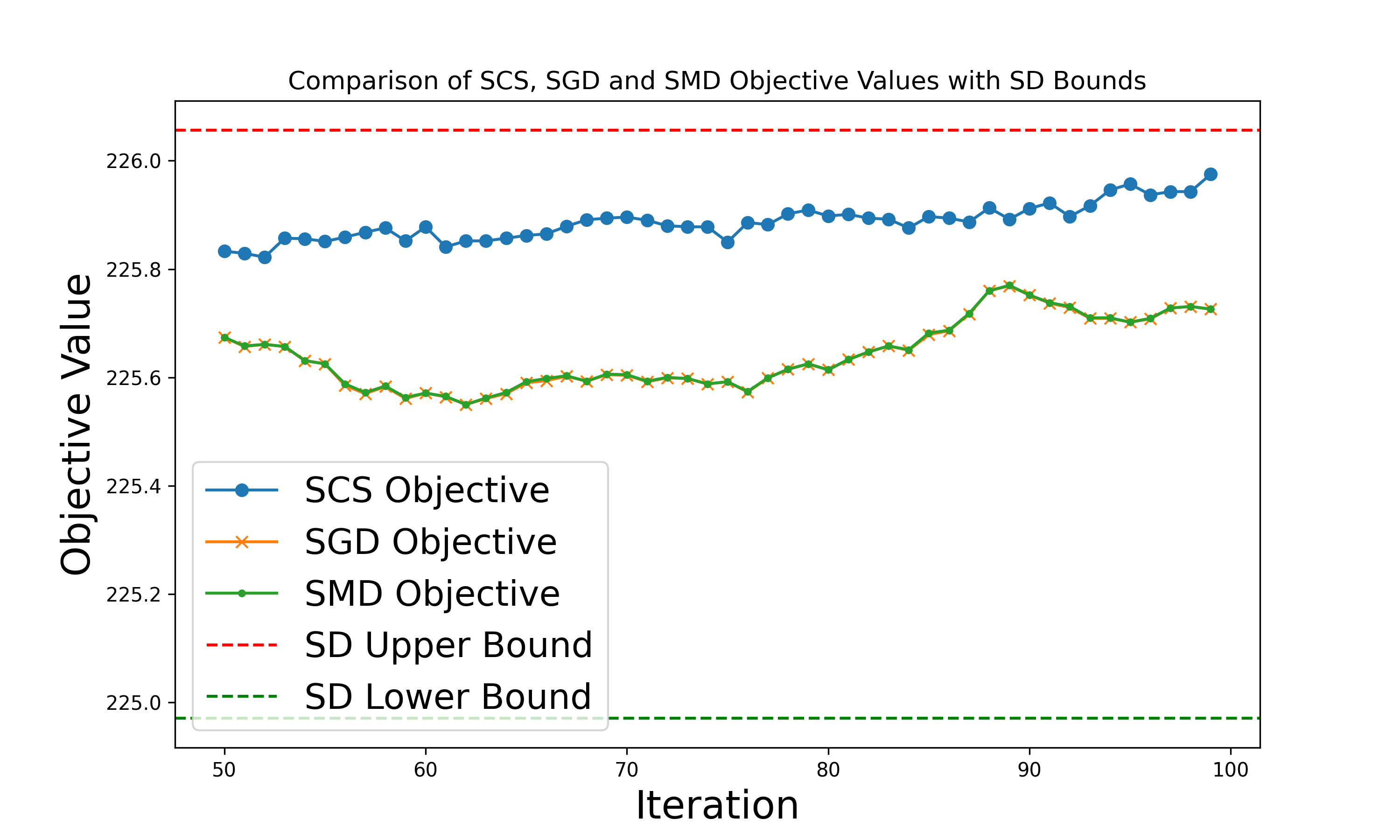}
\caption*{LandS3}
\end{minipage}
\begin{minipage}[t]{0.45\textwidth}
\centering
\includegraphics[width=7cm]{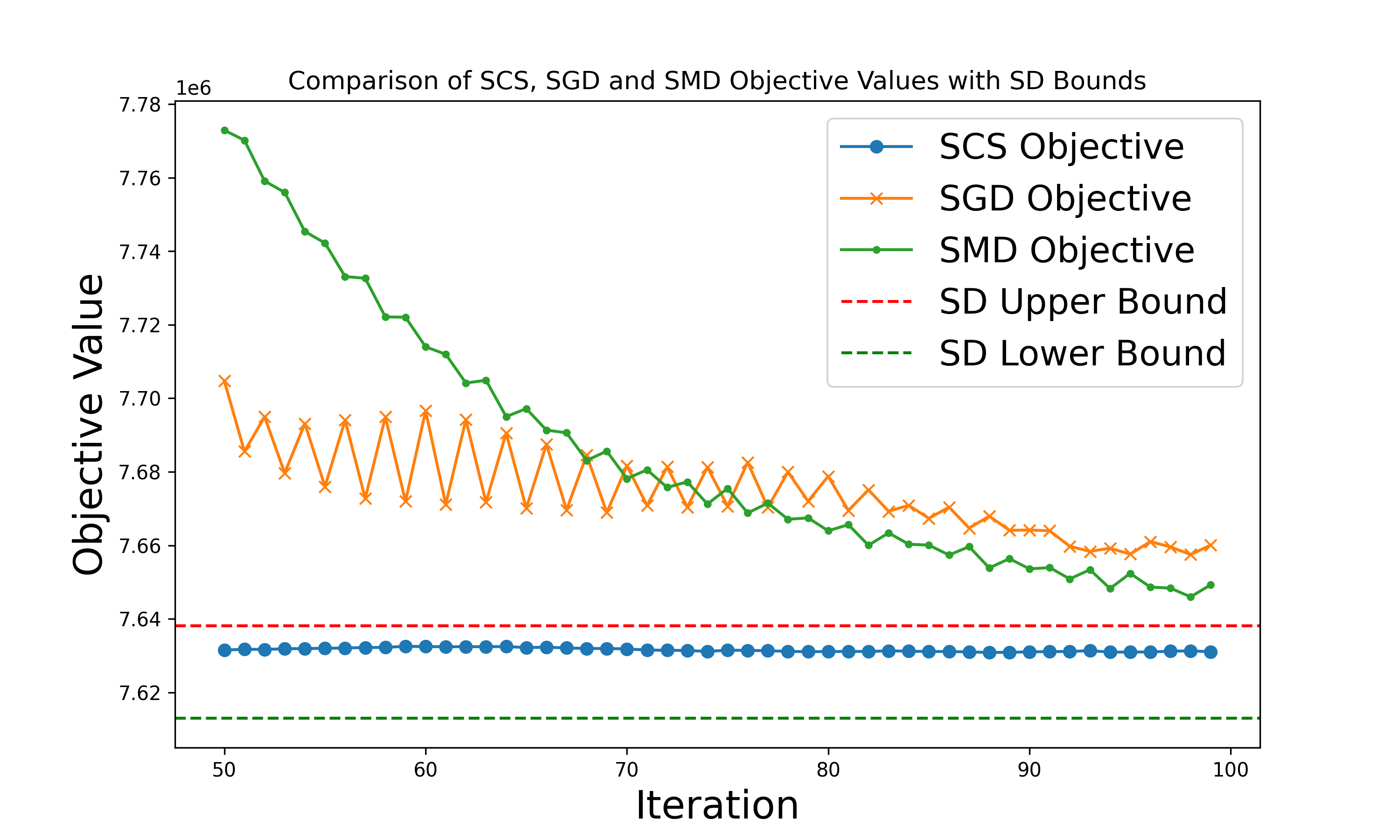}
\caption*{lgsc}
\end{minipage} \\

\begin{minipage}[t]{0.45\textwidth}
\centering
\includegraphics[width=7cm]{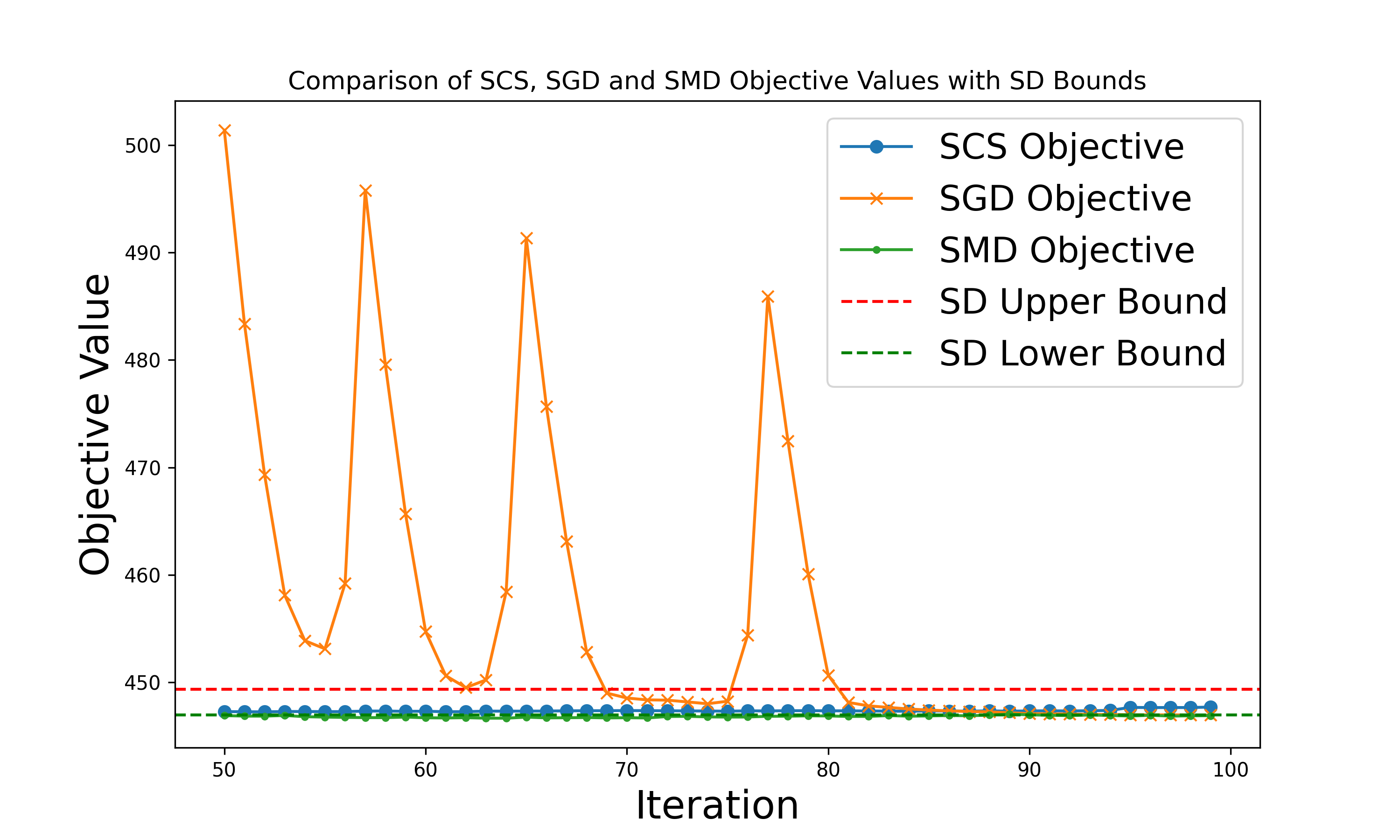}
\caption*{pgp2}
\end{minipage}
\begin{minipage}[t]{0.45\textwidth}
\centering
\includegraphics[width=7cm]{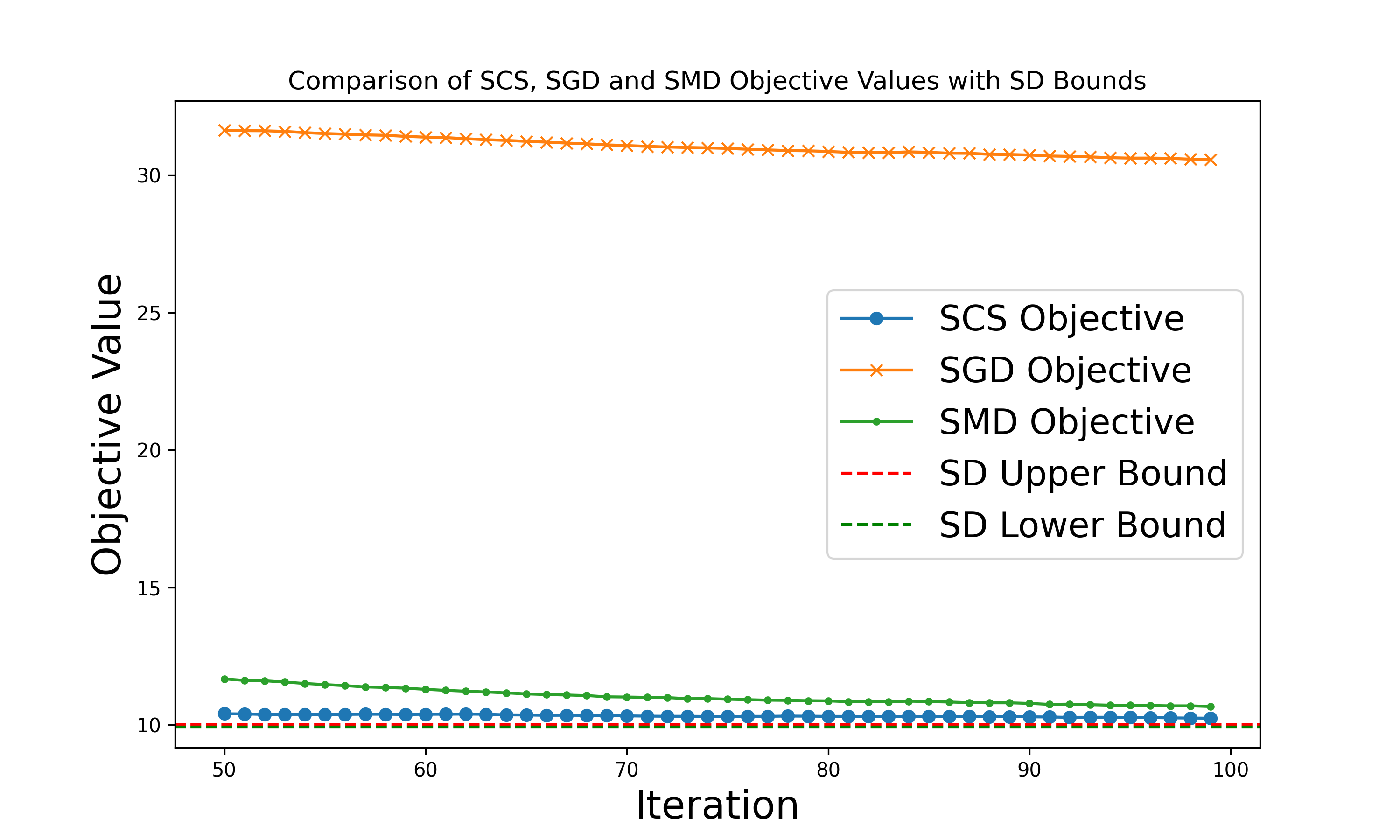}
\caption*{ssn}
\end{minipage}
\label{obj scs for sp last 50}
\end{figure}

\noindent \textbf{Remarks}: (i) The upper and lower bounds are derived from SD, which requires significantly more time to compute compared to the other algorithms. However, because these bounds are obtained through multiple replications, they provide a highly reliable estimate of the objective value, making them an excellent benchmark; (ii) Regarding the objective function value, SCS consistently yields lower values than the SGD and SMD algorithms, and it also converges faster; (iii) The SCS algorithm follows the principle of optimization, featuring a line-search procedure that automatically determines the step size. In contrast, other algorithms, such as SMD, require prior knowledge of the bound of the subgradient norm to set the step size, which is often impractical in many real-world applications.

\section{Conclusion}\label{Con}
Our approach leads to a class of online algorithms that go beyond first-order approximations. It incorporates several features of Wolfe's method (e.g., stability and a good convergence rate) as well as the online aspects of SGD, which promote good computational times. In contrast to the SGD algorithm, we find that the optimization performance of the SCS algorithm is more reliable and consistently provides lower objective values (see Figures \ref{obj scs for sp last 50}). It appears that such reliability may be difficult to achieve using first-order methods without additional fine-tuning efforts.\\
\bibliographystyle{abbrv}
\bibliography{main}   


\end{document}